\documentclass{amsart}

\usepackage[utf8]{inputenc}
\usepackage{amsmath}
\usepackage{amsfonts}
\usepackage{amssymb}
\usepackage{amsthm}
\usepackage[english]{babel}
\usepackage{fontenc}
\usepackage{graphicx}
\usepackage{fullpage}
\usepackage{color}
\usepackage{mathrsfs}
\usepackage[T1]{fontenc} 
\usepackage{mathtools}
\usepackage{graphicx}
\usepackage{stmaryrd}
\usepackage{multirow}
\usepackage[backend=biber,
            style=numeric,
            sorting=nyt,        % Sort by name -> year -> title
            giveninits=true,     % Use initials for first names
            maxbibnames=99,      % Show all authors
            isbn=false, doi=false, url=false % hide ISBN/DOI/URL
           ]{biblatex}

\DeclareNameAlias{sortname}{family-given}

\DeclareFieldFormat[article]{title}{#1}
\DeclareFieldFormat[inproceedings]{title}{#1}

\usepackage{csquotes}
\usepackage{dsfont}
\usepackage{tikz}
\usepackage{caption}
\usepackage{float}
\usepackage{algorithm}
\usepackage{verbatim}
\usepackage{algpseudocode}
\usepackage{caption}
\usepackage{subcaption}

\usetikzlibrary{calc}

\usepackage[normalem]{ulem}

\newtheorem{defi}{Definition}[section]

\newtheorem{rem}[defi]{Remark}

\newtheorem{prop}[defi]{Proposition}

\numberwithin{equation}{section}

\newcommand{\R}{\mathbb R}
\newcommand{\C}{\mathbb C}
\newcommand{\N}{\mathbb N}

\newcommand{\Hl}{\mathrm{H\mkern-2mu\ell}}

\renewcommand{\leq}{\leqslant}
\renewcommand{\geq}{\geqslant}

\newcommand{\pinf}{+\infty}
\newcommand{\minf}{-\infty}

\newcommand{\indic}[1]{\mathds{1}_{#1}}
\newcommand{\equ}[1]{\underset{#1}{\sim}}

\allowdisplaybreaks

\DeclareMathOperator\supp{supp}

\usepackage{hyperref}

\title{Computing resonances of perturbed Schrödinger equations: application to Reissner--Norsdtröm--de Sitter black holes}

\author{Valentin Arrigoni, Geneviève Dusson}
\thanks{Université Marie et Louis Pasteur, CNRS, LmB (UMR 6623), F-25000 Besançon, France.\\
Corresponding author: \href{mailto:valentin.arrigoni@math.cnrs.fr}{valentin.arrigoni@math.cnrs.fr}}

\date{}

\begin{document}

\maketitle

\begin{abstract}
We present a numerical method for computing resonances of one-dimensional Schrödinger equations perturbed by a compactly supported potential, via finding zeros of the Wronskian associated with Jost solutions of the reference equation, computed through the resolution of Cauchy problems. 
All resonances located in a given domain are found efficiently using a deflated Newton algorithm. A key ingredient of the method is the choice of reference potential for which Jost solutions are known, which removes spurious resonances often encountered numerically. We test this method on three types of reference potentials and perturbations thereof: Pöschl--Teller potentials, exponentially decaying potentials, and potentials associated with Reissner--Nordström--de Sitter black holes. In particular we study the impact of perturbations on the resonances, and the stability of small resonances under perturbation.  
As an illustration, we use the method to numerically study the strong cosmic censorship hypothesis.
\end{abstract}

\section{Introduction}
\label{Intro}
Resonances appear in many different scientific domains such as spectral analysis, quantum mechanics or general relativity. Following Zworski~\cite{ZworskiResonances}, they correspond to a notion of eigenvalues to systems which allow energy to scatter to infinity. 
In the specific context of general relativity and black holes, resonances also called quasinormal modes are strongly linked with gravitational waves. Their detection in 2015 by LIGO~\cite{GravWave} has made these resonances highly relevant for study within the community. From a mathematical perspective, they can be defined as poles of a 
resolvent operator~\cite[Definition 1]{ZworskiResonances}.
In this article, we are interested in the computation of such resonances for one-dimensional Schrödinger equations. From a physical point of view, the most important resonances are those of smallest imaginary part
as they 
decay more slowly and are thus more likely to be detected by sensors.

So far, many theoretical results have been obtained. In 1987, Zworski~\cite{Zworski87} established an asymptotics for the number of resonances inside a disk of growing size, valid for compactly supported potentials. He also proved that resonances are asymptotically located on two logarithmic branches, depending among other things, on the support and regularity of the potential.
Guillopé and Zworski~\cite{GuillopeZworski95,GuillopeZworski}, working on hyperbolic manifolds, obtained results on the number of resonances inside a disk of a given size. Hitrik~\cite{Hitrik} proposed estimates for the number of resonances in a disk for the class of one-dimensional exponentially and super-exponentially decaying potentials. 
Stepin and Tarasov \cite{StepTara07} refined the asymptotics for compactly supported potentials on the line and provided results on the location of resonances for super-exponentially decaying potentials in~\cite{StepTara09}. On the half-line, considering a compactly supported perturbation of an exponentially decaying potential, Borthwick, Boussaïd, and Daudé \cite{BBD} established asymptotics for the resonances. More recently, the first author~\cite{VA} derived asymptotic formulas for resonances when the potential is a compactly supported perturbation of a Pöschl--Teller potential.
A specific field where resonances have been studied is general relativity and especially black holes.
Bachelot and Motet-Bachelot~\cite{BMB} studied resonances for Schwarzschild black holes.
For the same black holes, Sá Barreto and Zworski~\cite{SaBaZwo} gave approximation formulas for resonances when the angular momentum increases.
Dyatlov~\cite{Dyatlov} worked on Kerr--de Sitter black holes, providing a definition for resonances.

Concerning quasinormal modes of Reissner--Nordström--de Sitter black holes,
Iantchenko \cite{Iantchenko} established asymptotic formulas for the resonances. 
Results for more general black holes (Kerr–Newman–de Sitter) can also be found in~\cite{Suzuki}. 
We mention the book of Dyatlov and Zworski~\cite{DyaZwo} which summarizes many results on resonances, in particular for black holes.
Since asymptotic analysis only brings information on the resonances with large modulus, these previous works cannot be used to locate the smallest ones. As far as we are aware, the only theoretical work on the location of small resonances is~\cite{HinXie} for de Sitter--Schwarzschild black holes in the limit where the mass parameter tends to zero. In general, obtaining precise information on the smallest resonances seems difficult to access via theoretical analysis.
This is our main motivation for proposing a numerical approach to compute them.

The problem of the numerical computation of resonances has already been studied in the past, but in a different framework. Bindel and Zworski~\cite{BZ} developed an approach to resonance computation for compactly supported potentials, via the resolution of an eigenvalue problem. 
Another well-developped numerical methods is called complex scaling~\cite{Datchev,Cerioni}.
More recently, Duchemin and coauthors~\cite{Levitt} developed a generalised method to compute resonances induced by localised defects in crystals, using a reformulation of the problem as an integral equation (Birman-Schwinger), and showed that their method performs better than complex scaling.
Regarding the numerical resonances associated with black holes, several methods have been proposed.
Kokkotas and Schmidt \cite{KokkoSch} computed resonances for stars and black holes using the Wentzel–Kramers–Brillouin (WKB) method. The Frobenius method was also used, see e.g.~\cite{Zhi,Leaver85,Nollert}.
It is also possible to exploit the link between quasinormal modes of potential barriers and bound states of the corresponding potential wells to compute quasinormal modes~\cite{Volkel,Mashh}. 
Other works perform pseudospectrum calculations using Chebyshev spectral method~\cite{JLJ,Aimer,Destounis}.
Note that most of these works use Pöschl–Teller potentials as a toy model for real black holes, since the resonances for this class of potentials are exactly known.

In this work, we propose a new approach to compute resonances, exploiting the fact that we are considering a one-dimensional scattering problem. Indeed in this specific setting, we are able to define resonances as zeros of a Wronskian $w$, instead of poles of the meromorphic continuation of the resolvent in the lower half-plane. 
Thus we then determine the zeros of  of a complex function, which we do using a deflated version of the classical Newton’s method~\cite{Farrell}. We consider perturbations with compact support of a reference potential, and we are free to  choose the reference potential, as long as we know the corresponding Jost solutions. 
This avoids numerical instabilities compared to using the null potential as a reference.
Moreover, the deflated Newton's method allows to compute all resonances in a given domain, the total number of resonances being computed beforehand using a contour integration, hence we are not limited to the computation of only a few resonances. 
Even though our method is currently limited to the one-dimensional case and higher-dimensional cases with symmetries, we hope to provide an alternative definition of resonances in higher dimensions in order to apply this approach to a larger class of problems. This is left for future work.

The article is organised as follows. 
In Section~\ref{Sec1}, we present the proposed numerical method to compute the resonances.
Section~\ref{Sec2} presents Jost solutions for three different reference potentials: Pöschl--Teller, exactly exponentially decaying and Reissner--Nordström--de Sitter black holes potentials.
Section~\ref{Sec3} focuses on numerical experiments. We begin by analysing several results related to the deflated Newton's method, and observe stability of small resonances for perturbations of the considered reference potentials.
We conclude by numerically studying the strong cosmic censorship hypothesis.
\section{Numerical method for computing resonances}
\label{Sec1}
\subsection{Mathematical setting}
\label{MathsSetting}

In the following we use the notation $\N = \{1,2,\ldots \}$ and $\N_0 = \{0,1,2,\ldots\}$.

Let $V:\R\rightarrow\R$ be a function in $L^1_1(\R)$,
that is
\[
\int_\R \left( 1+|x| \right)|V(x)| \, \mathrm{d}x < \pinf.
\]
Given $z \in \C$, we consider the following Schrödinger equation:
\begin{equation}
-y''+Vy = z^2y, \quad \text{in} \; \R.
\label{Premiere}
\end{equation}
Among all solutions to \eqref{Premiere}, we call Jost solution \cite[Chapter 6]{Levitan} and denote by $f^+$ (respectively $f^-$) the unique solution to~\eqref{Premiere} satisfying at $\pinf$ (respectively $\minf$) the prescribed asymptotics: 
\begin{equation}
f^\pm(x,z) \equ{x \to \pm \infty} e^{\pm ixz}.
\label{Asymptotics}
\end{equation}
More generally, we still call Jost solutions, the solutions to second-order ordinary differential equations satisfying~\eqref{Asymptotics}.
From these solutions, we can define a Wronskian, also called Jost function,
\[
\begin{array}{ccccc}
w & : & \C & \to & \C \\
 & & z & \mapsto & [f^-(x_0,z),f^+(x_0,z)],\\
\end{array}
\]
where $x_0 \in \R$ and $[f,g] := fg'-f'g$. Since \eqref{Premiere} is a second-order linear ordinary differential equation without a first-order term, the Wronskian between the two solutions $f^+$ and $f^-$ indeed does not depend on the first variable $x_0$. For a general potential $V \in L^1_1(\R)$, the Wronskian $w$ is always defined on the closed upper-half plane, and is analytic on the open upper-half plane. 
We consider throughout the remainder of this work, $V$ such that the Jost solutions (and so $w$) can be extended meromorphically
to the lower half-plane. An example of such potentials are Pöschl--Teller potentials \cite[Section~2, (14)]{VA}.
We can now express eigenvalues and resonances in terms of this function $w$, denoting by $\Im(z)$ the imaginary part of $z\in\C$. 
\begin{defi}
    \begin{itemize}
    \item A complex number $z^2$, for $z = ik$ where $k>0$, is called an eigenvalue if $z$ is a zero of the function $w$ in the open upper half-plane $\C^+ := \{ z \in \C ~|~ \Im(z)>0 \}$.
    \item A complex number $z^2$ is called a resonance if $z$ is a zero of the meromorphic continuation of $w$ in the open lower half-plane $\C^- := \{ z \in \C ~|~ \Im(z)<0 \}$.
    \end{itemize}
    \label{EigRes}
\end{defi}
No matter whether $z^2$ is an eigenvalue or a resonance, the fact that $z$ is a zero of $w$ implies that, for such $z \in \C$, the two Jost solutions $f^+(\cdot,z)$ and $f^-(\cdot,z)$ are colinear. Using the conditions \eqref{Asymptotics}, a direct calculation proves that when $z^2$ is an eigenvalue, the two Jost solutions lie in $L^2(\R)$ and are colinear eigenvectors associated with this eigenvalue. By contrast, if $z^2$ is a resonance, then the Jost solutions are no longer in the space $L^2(\R)$, again due to the asymptotics \eqref{Asymptotics}. Hence, resonances can be understood as eigenvalues without eigenvectors. Finally, $w(z)\neq 0$ for $z\in \R\backslash \{0\}$~\cite[Lemma 1 (i) and (iv)]{Bledsoe}, and when $w(0)=0$, we say that there exists a half-bound state. 

In order to locate resonances, we will solve the equation $w(z)=0$ in a given domain  $\Omega\subset\C^-$. To do so, we first explain how to evaluate the Wronskian $w(z)$.
\subsection{Evaluation of the Wronskian}
In this section, we consider real-valued potentials $V$ 
expressed
as the sum of two functions: $V= V_0 + q$. The first function $V_0 \in L^1_1(\R)$ is called the reference potential. 
Examples of such potentials $V_0$ are Pöschl--Teller potentials: $x \mapsto \frac{\lambda}{\cosh^2(x)}$, for $\lambda \in \R \backslash \{0\}$. The second function $q$ is a real, integrable and compactly supported function such that its support satisfies $\supp{q} \subset [\alpha,\beta]$, $\alpha,\beta \in \R$.
We denote by $f^\pm_0$ the two Jost solutions associated with the following Schrödinger equation
\[
-y''+ V_0y = z^2y, \quad \text{in } \R.
\]
In order to numerically compute Jost solutions associated with the potential $V$, we first characterize them by a Cauchy problem instead of an asymptotics. 
\begin{prop}\label{Prop1}
    Let $V$ be a potential such that $V = V_0 + q$ for $V_0 \in L^1_1(\R)$ and $q$ a real-valued, integrable and compactly supported function with its support in $[\alpha,\beta]$, $\alpha,\beta \in \R$. Let $z \in \C$ be such that the Jost solution $f^+_0$ associated with \eqref{Premiere} for $V=V_0$ is well-defined. Then, the Jost solution $f^+$ associated with the Schrödinger equation \eqref{Premiere} satisfies the Cauchy problem
    \begin{equation}
\left\{
\begin{array}{l}
-y'' + Vy = z^2y, \quad \text{in}~[\alpha,\beta]\\
y(\beta) = f^+_0(\beta,z)\\
y'(\beta) = (f^+_0)'(\beta,z).
\end{array}
\right.
\label{CP1}
\end{equation}
The notation $(f^+_0)'$ stands for the derivative of the function $f^+_0$ with respect to the first variable. 
\end{prop}
\begin{proof}
Let $z \in \C$ be such that the Jost solution $f^+_0$ associated with \eqref{Premiere} for $V=V_0$ is well-defined.
{For $x \geq \beta$, $V(x) = V_0(x)$. So, on $[\beta,\pinf)$, the uniqueness of Jost solutions proves that $f^+ = f_0^+$.}
    Finally, since $f^+$ is a solution of \eqref{Premiere} which satisfies both $f^+(\beta,z) = f_0^+(\beta,z)$ and $(f^+)'(\beta,z) = (f^+_0)'(\beta,z)$, then by the Cauchy--Lipschitz theorem, it is the unique solution of \eqref{CP1}.
    \end{proof}
A similar result holds for the other Jost solution $f^-$.
\begin{prop}
    Let $V$ be a potential such that $V = V_0 + q$ for $V_0 \in L^1_1(\R)$ and $q$ a real-valued, integrable and compactly supported function with its support in $[\alpha,\beta]$, $\alpha,\beta \in \R$. 
    Let $z \in \C$ be such that the Jost solution $f^-_0$ associated with \eqref{Premiere} for $V=V_0$ is well-defined.
    Then, the Jost solution $f^-$ associated with the Schrödinger equation \eqref{Premiere} satisfies the Cauchy problem
\begin{equation}
\left\{
\begin{array}{l}
-y'' + Vy = z^2y,\quad \text{in} \; [\alpha,\beta]\\
y(\alpha) = f^-_0(\alpha,z)\\
y'(\alpha) = (f^-_0)'(\alpha,z).
\end{array}
\right.
\label{CP2}
\end{equation}
\end{prop}
Since the Wronskian $w$ does not depend on the real variable $x_0$, its computation only requires the Jost solutions and their derivatives at a specific point. To evaluate $w$, we choose $x_0=\beta$. This way we only have to solve one Cauchy problem instead of two since $f^+(\beta,z) = f^+_0(\beta,z)$. We present the evaluation of the Wronskian in Algorithm~\ref{Algow}.
\begin{algorithm}[H]
\caption{Computation of the wronskian $w$}
\label{Algow}
\begin{algorithmic}[1]
\State Input: $z\in\C:$ evaluation point, $V_0$: reference potential, $q$: perturbation, $\alpha,\beta$: endpoints of the support of $q$.
\State Solve the Cauchy problem \eqref{CP2} to find $f^-$.
\State Output: value of the Wronskian at $z$: $(f^+_0)'(\beta,z)f^-(\beta,z) - (f^-)'(\beta,z)f^+_0(\beta,z)$.
\end{algorithmic}    
\end{algorithm}
\subsection{Location of the resonances via Newton's deflated method}

Once we are able to evaluate $w$, we find its zeros by using a deflated Newton's method~\cite{Farrell}. 
The idea is to apply Newton's method several times with modified functions in order to avoid converging several times to the same root. The modified function is the product of the original function $w$ and a specific deflation factor diverging near the already discovered zeros and tending to 1 far from them. Denoting by $Z$ the set of already found zeros, the deflation factor used in this work is
\begin{equation}
\forall u\in \C, \quad M(u) = \left\{
    \begin{array}{ll}
        1 & \mbox{if } Z = \emptyset \\
       \displaystyle \prod_{z \in Z} \left(\frac{1}{|u-z|}+1 \right) & \mbox{otherwise}
    \end{array}
\right.
\label{DeflaFact}
\end{equation}
Note that the factor $|u-z|$ is chosen to capture the possible higher multiplicity of the zeros. In the case of a zero with multiplicity $m$, the deflated Newton method will converge $m$ times to this zero. 
This would not be the case with higher powers of $|u-z|$.
We detail Newton's deflated method in Algorithm~\ref{Algo1}. In practice, the deflation only changes the step size and not the descent direction compared to the standard Newton's method.
\begin{algorithm}[H]
\caption{Newton's deflated method}
\label{Algo1}
\begin{algorithmic}[1]
\State Input: $w$: function whose zeros are to be determined, $Z$: known zeros of $w$, $z$: initial guess, $N$: maximum number of iterations, $\varepsilon$: tolerance.
\For {$i$ from 1 to $N$}
\State $\delta_f = -w'(z)/w(z)$.
\State $p = \delta_f \times M'(z)$.
\State $\tau = 1 + \frac{p/M(z)}{1-p/M(z)}$.
\State $z = z + \tau  \delta_f$
\If {$|w(z)| \leq \varepsilon$}
\State break
\EndIf
\EndFor
\State Output: $z$, a new zero of the function $w$.
\end{algorithmic}    
\end{algorithm}

In practice we are looking for the zeros of $w$ in a specific domain $\Omega$. 
First, due to the fact that $\overline{w(z)} = w(-\overline{z})$ \cite[Lemma 1.1]{VA}, resonances are symmetric with respect to the imaginary axis, so we only focus on resonances with a non-negative real part.
Second, we  determine a priori how many roots lie in $\Omega$, that is before starting to locate them, using the residue theorem.
Let $\gamma$ be a positively oriented simple closed curve in $\C^-$ enclosing~$\Omega$, $Z_\gamma$ and $P_\gamma$ be respectively the set of zeros and poles of $w$ in the interior of $\gamma$. 
Then an application of the residue theorem on $\gamma$ to the function $\frac{w'}{w}$ gives
\[
\frac{1}{2i\pi} \int_{\gamma} \frac{w'(z)}{w(z)} \, \mathrm{d}z = \sum_{z \in Z_\gamma \cup P_\gamma} \textrm{Res} \left(\frac{w'}{w},z\right)  = \# Z_\gamma - \# P_\gamma.
\]
In the following we typically work with rectangular domains, for which the path is shown in Figure~\ref{fig:Rectangle}.
{For the reference potential we will later be interested in, the location and multiplicity of the poles $P_\gamma$ are well-known (see Remarks~\ref{rem:poles1},~\ref{rem:poles2} and~\ref{rem:poles3}), so we can easily determine $\# Z_\gamma$.}
\begin{figure}[H]
\begin{center}
\begin{tikzpicture}[scale=0.5]

  \draw[->] (-7,0) -- (7,0) node[right] {Re};
  \draw[->] (0,-5) -- (0,2) node[above] {Im};

  \draw[thick,blue] 
    (-5,0) -- (5,0) -- (5,-3) -- (-5,-3) -- cycle;

  \draw[->,thick,blue] (0,0) -- (-3.5,0);
  \draw[->,thick,blue] (5,-3) -- (5,-1);
  \draw[->,thick,blue] (0,-3) -- (3.5,-3);
  \draw[->,thick,blue] (-5,0) -- (-5, -2);

  \filldraw[black] (-5,0) circle (2pt) node[above left] {$-R$};
  \filldraw[black] (5,0) circle (2pt) node[above right] {$R$};
  \filldraw[black] (5,-3) circle (2pt) node[below right] {$R-ir$};
  \filldraw[black] (-5,-3) circle (2pt) node[below left] {$-R-ir$};
\end{tikzpicture}

\captionsetup{justification=centering}
\end{center}
\caption{Illustration of the path along which the integral is computed.}
\label{fig:Rectangle}
\end{figure}
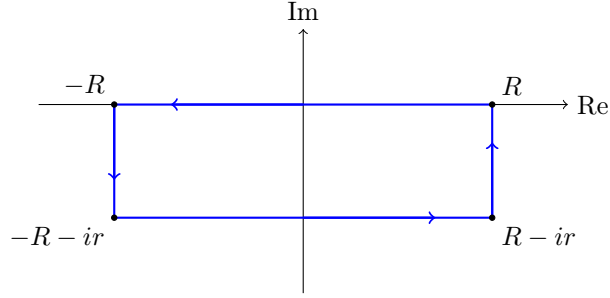

Finally, to obtain the collection of zeros, we apply the deflated Newton's method several times until we have found the correct number of zeros as detailed in Algorithm~\ref{AlgoLocZ}. 
\begin{algorithm}[H]
\caption{Location of all zeros}
\label{AlgoLocZ}
\begin{algorithmic}[1]
\State Input: $w$: function whose zeros are to be determined, $\Omega$: spatial domain, $N$: number of zeros in $\Omega$.
\State $Z:= \emptyset$
\While {$n < N$}
\State Choose $z_0$ randomly (uniformly) in the domain $\Omega$.
\State Apply Newton's deflated method (Algorithm \ref{Algo1}), starting at $z_0$, {with $Z$ as set of known zeros}.
\If {Newton's deflated method provides a solution $z$}
\State $Z = Z \cup \{ z\}$
\State $n = n + 1$
\EndIf
\EndWhile
\State Output: $Z$: list of all roots of $w$ in $\Omega$.
\end{algorithmic}    
\end{algorithm}
With this algorithm, we are able to find resonances for different compactly supported perturbation of reference potentials, provided that the Jost solutions for the corresponding reference potentials are known. 

\section{Jost functions for three reference potentials of interest}
\label{Sec2}

In this section, we present different reference potentials and corresponding Jost solutions. 
Working with specific reference potentials will avoid some numerical instabilities as presented in Section~\ref{ImpChoicePot}.
\subsection{Pöschl--Teller potentials}

For now we consider a reference potential for $\lambda \in \R\backslash\{0\}$
\begin{equation}
\label{eq:PT}
    V_0 : x \mapsto \frac{\lambda}{\cosh^2(x)}.
\end{equation}
This potential was introduced for the first time by G. Pöschl and E. Teller~\cite{PT} in their study of the quantum anharmonic oscillator, and can be seen as toy models for real black holes.
The two corresponding Jost functions $f^\pm_0$ can be explicitly written by means of the hypergeometric functions ${}_2F_1$ \cite[Definition 4]{VA}.
\begin{prop}{(See \cite[(12),(13)]{VA})}
Let $V_0$ be a Pöschl--Teller potential as defined in~\eqref{eq:PT}.
Let $z \in \C \backslash (-i\N)$.
Then, for $x \in \R$, the two Jost solutions associated with~\eqref{Premiere} are given by:
\begin{equation*}
    f^+_0(x,z) = (2\cosh(x))^{iz} {}_2F_1\left(a(z),b(z),c(z);\frac{1}{1+e^{2x}}\right),
    \label{Jost+}
\end{equation*}
and
\begin{equation*}
    f^-_0(x,z) = (2\cosh(x))^{iz} {}_2F_1\left(a(z),b(z),c(z);\frac{e^{2x}}{1+e^{2x}}\right),
    \label{Jost-}{}
\end{equation*}
with 
\begin{align*}
a(z) & := \frac{1}{2} - i z + \sqrt{\frac{1}{4}- \lambda}, \qquad 
b(z)  := \frac{1}{2} - i z - \sqrt{\frac{1}{4}- \lambda}, \qquad
c(z)  := 1 - i z.
\end{align*}
\end{prop}
\begin{rem}[Poles]
    \label{rem:poles1}
    The hypergeometric functions $z \mapsto {}_2F_1\left(a(z),b(z),c(z);\zeta \right)$ have poles for $c(z) \in -\N_0$, so for $z \in -i\N$ \cite[Section 2]{VA}.
    It is then easy to see that the Jost solutions $f^\pm_0$ multiplied by $1/\Gamma(1-iz)$ are entire functions.
    Similarly, as detailed in~\cite[(34)-(35)]{VA}, the Jost solutions $f^\pm$ multiplied by $1/\Gamma(1-iz)$ denoted by 
    $\mathbf{f}^\pm$
    are also entire.
    From this we obtain that $[\mathbf{f}^-,\mathbf{f}^+]$ is also an entire function, and since
    \[
     w(z) = \Gamma(1-iz)^2 
    [\mathbf{f}^-(x,z),\mathbf{f}^+(x,z)],
    \]
    the Wronskian $w$ has poles for $z$ exactly in $-i\N$, see~\cite[Lemma 2]{VA}.
\end{rem}
\subsection{Exactly exponentially decaying potentials.}
\label{Eedp}

We now consider a generalisation of Pöschl--Teller potentials, which can be seen as refined toy models for the analysis of Reissner-Nordström--de Sitter black holes. Indeed these black holes have a corresponding potential expressed as a series in a neighbourhood of $\pm \infty$, and keeping only the first term of this series, we obtain exactly exponentially decaying potentials at $\pm \infty$.
Therefore we assume that, for the potential $V_0 :\R \rightarrow \R$, there exist $R_- < 0 < R_+$, $\kappa_->0$, $\kappa_+<0$ and $\mu_\pm >0$ such that
\begin{equation} \label{kappa+}
    V_0(x) = \mu_+ e^{2\kappa_+x}, ~x \geq R_+,
\end{equation}
and 
\begin{equation} \label{kappa-}
    V_0(x) = \mu_- e^{2\kappa_-x}, ~x \leq R_-.
\end{equation}
\begin{rem}
The $\kappa_\pm$, called surface gravity in the context of general relativity, satisfy the misleading property $\kappa_->0$ and $\kappa_+<0$.
\end{rem}
 Here the Jost solutions can be expressed using Bessel functions~\cite[Section 10.2]{NIST:DLMF}. 
\begin{prop}
Let $R_- < 0 < R_+$, $\kappa_- > 0$, $\kappa_+ < 0$, $\mu_\pm > 0$ and $V_0$ be an exponentially decaying potential satisfying \eqref{kappa+} and \eqref{kappa-}. Let $z \in \C$ such that $1+\frac{iz}{\kappa_+} \in \C \backslash (-\N_0)$ 
 and $1-\frac{iz}{\kappa_-} \in \C \backslash (-\N_0)$. Then the two Jost solutions associated with~\eqref{Premiere} satisfy
\begin{equation}
\forall R_+ \leq x, \quad {f}^+_0(x,z) := J_{\frac{iz}{\kappa_+}}\left(\frac{i\sqrt{\mu_+}}{\kappa_+} e^{\kappa_+x}\right)
\Gamma\left(1+\frac{iz}{\kappa_+}\right) \left( \frac{2 \kappa_+}{{i\sqrt{\mu_+}}} \right)^{\frac{iz}{\kappa_+}},
\label{Eq:JostFuncBessel+}
\end{equation}
and
\begin{equation}
\forall x \leq R_-, \quad
{f}^-_0(x,z) := J_{\frac{-iz}{\kappa_-}}\left(\frac{i\sqrt{\mu_-}}{\kappa_-} e^{\kappa_- x}\right)\Gamma\left(1-\frac{iz}{\kappa_-}\right) \left( \frac{2 \kappa_-}{{i\sqrt{\mu_-}}} \right)^{\frac{-iz}{\kappa_-}},
\label{Eq:JostFuncBessel-}
\end{equation}
where $J$ are the Bessel functions of first kind and $\Gamma$ is the Gamma function.
\end{prop}
\begin{proof}
On $[R_+,\pinf)$, \eqref{Premiere} becomes
\begin{equation}
-y''(x) + \mu_+ e^{2\kappa_+x}y(x)=z^2y(x).
\label{Deuxi}
\end{equation}
We follow the sketch of proof presented in~\cite[Lemma 7]{BBD}. Applying the following changes of variables, $x(t) = \frac{1}{\kappa_+}\ln\left( \frac{t \kappa_+}{\sqrt{\mu_+}}\right)$ 
for $t \in \R_-$ and $v(t) = y(x(t))$, straightforward calculations show that $y'(x(t)) = t\kappa_+v'(t)$ and $y''(x(t)) = \kappa_+^2t (t v''(t) + v'(t))$. Then, by replacing these values in~\eqref{Deuxi}, $v$ satisfies
\[
t^2v''(t)+tv'(t)-\left(t^2 + \left( \frac{iz}{\kappa_+} \right)^2\right)v(t)=0.
\]
A solution of this second-order differential equation is a modified Bessel function $v(t) = I_{\frac{iz}{\kappa_+}}(t)$~\cite[Section 10.25 \unskip]{NIST:DLMF}. Using the link between Bessel and modified Bessel functions~\cite[10.27]{NIST:DLMF}, 
we have the following expression for $y$: 
\[
y(x) = J_{\frac{iz}{\kappa_+}}\left( \frac{i\sqrt{\mu_+}}{\kappa_+} e^{\kappa_+ x} \right).
\]
Since for $\zeta \to 0$ , $J_{\frac{iz}{\kappa_+}}(\zeta) \sim \frac{(\zeta/2)^{\frac{iz}{\kappa_+}}}{\Gamma\left(\frac{iz}{\kappa_+} +1\right)}$ \cite[(10.7.3)]{NIST:DLMF}, and taking $\zeta = \frac{i\sqrt{\mu_+}}{\kappa_+} e^{\kappa_+ x}$, with $x\rightarrow +\infty$, we see that $y$ needs to be multiplied 
by a factor $\Gamma\left(1+\frac{iz}{\kappa_+}\right) \left( \frac{2 \kappa_+}{{i\sqrt{\mu_+}}} \right)^{\frac{iz}{\kappa_+}}$
to satisfy the asymptotics~\eqref{Asymptotics}.

Similarly, working on $(\minf,R_-]$ leads to~\eqref{Eq:JostFuncBessel-} for $f^-_0$ .
\end{proof}
\begin{rem}[Poles]
    \label{rem:poles2}
    Since 
    $z \mapsto J_{\frac{iz}{\kappa_+}}\left(\frac{i\sqrt{\mu_+}}{\kappa_+} e^{\kappa_+x}\right)$ and $z \mapsto J_{\frac{-iz}{\kappa_-}}\left(\frac{i\sqrt{\mu_-}}{\kappa_-} e^{\kappa_- x}\right)$ are analytic in $\C$ for some $x \in \R$, the Jost solutions ${f}^+_0$ and ${f}^-_0$ respectively have possible poles for $z \in i\kappa_+\N$ and  $z \in -i\kappa_-\N$.
    With the same reasoning as in Remark~\ref{rem:poles1}, the associated Wronskian when the reference potential satisfies~\eqref{kappa+} and~\eqref{kappa-} has possible poles of multiplicity one in $(i\kappa_+\N) \cup (-i\kappa_-\N)$.
\end{rem}
\subsection{Reissner--Nordström--de Sitter black holes}

As already mentioned, Pöschl--Teller and exactly exponentially decaying potentials are toy models in the analysis of some black holes. In this section, we study the case of Reissner--Nordström--de Sitter black holes, and start with some reminders before providing the Jost solution expressions in Proposition~\ref{RNdSProp}, following~\cite{Hatsuda}

We consider a Reissner--Nordström--de Sitter black hole, \textit{i.e.} a spherical, massive, electrically charged black hole lying in a universe undergoing accelerated expansion. The associated metric function $F$ is given by
\begin{equation} \label{MetricFunctionF}
F(r) := 1 - \frac{2M}{r} + \frac{Q^2}{r^2} - \frac{\Lambda}{3}r^2, \quad 
r \neq 0,
\end{equation}
where $M>0$ is the mass of the black hole, $Q \in \R$ its electrical charge and $\Lambda >0$ the cosmological constant. Let us define the constants $\mathfrak{R}$, $D$, $m_1$, $m_2$, $M_1$ and $M_2$ by 
\begin{gather*}
    \mathfrak{R} := \frac{1}{\sqrt{2\Lambda}}, \quad D := 1-4Q^2\Lambda, 
    \quad m_1 := \mathfrak{R} \sqrt{1-\sqrt{D}}, 
    \quad m_2 := \mathfrak{R} \sqrt{1+\sqrt{D}}, \\
    M_1 := m_1-\frac{2}{3}\Lambda m_1^3, 
    \quad M_2 := m_2 -\frac{2}{3}\Lambda m_2^3.
\end{gather*}
Under the following assumptions on these constants detailed in \cite[Section 2]{Mokdad},
\begin{align}
    Q \neq 0, \quad  0 < \Lambda < \frac{1}{4Q^2}, \quad \text{and} \quad M_1 < M < M_2,
    \label{Mokdad}
\end{align}
one can prove that the function 
\begin{equation}\label{Eq:Delta(r)}
\Delta : r\in \R \mapsto \Delta(r) := r^2 F(r)
\end{equation}
has four real roots: $r_n < 0 < r_c < r_- < r_+$. The quantity $r_c$ is called the Cauchy horizon, $r_-$ is the event horizon and finally, $r_+$ is the cosmological horizon. Furthermore, on $(r_-,r_+)$, we have that $F >0$ \cite[Section 2]{Iantchenko}.
We then define $\kappa_*$ for $* \in \{n,c,-,+\}$ by 
\begin{equation} \label{kappaeq}
    \kappa_* = \frac{F'(r_*)}{2}.
\end{equation}
The parameters $\kappa_\pm$ are called surface gravities, and satisfy $\kappa_- >0$ and $\kappa_+ <0$. Note that they correspond to the parameters appearing in~\eqref{kappa+} and~\eqref{kappa-}. 

We now give the expression of the potential $V$ of~\eqref{Premiere} associated with Reissner--Nordström--de Sitter black holes~\cite[(7)]{Cardoso}, for $l \in \N_0$:
\begin{equation}
    \label{eq:pot_reissner}
    \forall x\in \R, \quad V(x) = F(r(x))\frac{l(l+1)}{r(x)^2} + F(r(x)) \frac{F'(r(x))}{r(x)} + \frac{2}{3}\Lambda F(r(x)),
\end{equation}
with $r(x)$ being the reciprocal function of 
\begin{equation}
    \label{x(r)}
x: r\in (r_-,r_+) \mapsto \frac{1}{2\kappa_n} \ln(r-r_n)+\frac{1}{2\kappa_c} \ln(r-r_c)+\frac{1}{2\kappa_-} \ln(r-r_-)+\frac{1}{2\kappa_+} \ln(r_+-r) \in \R.
\end{equation}
Note that the link between the Regge--Wheeler coordinate $x$ and the radial coordinate $r$ is given by $x'(r) = \frac{1}{F(r)}$~\cite[(2.3)]{Iantchenko}.

It is known that for a Reissner--Nordström--de Sitter black hole, on a neighbourhood of $\pinf$, the potential $V$ defined in~\eqref{eq:pot_reissner} can be written as a power series~\cite[Proposition A.1.]{Kehle} for some real numbers $C^+_m, \; m\ge 1$,
\begin{equation} \label{Vl+}
V(x) = \sum_{m=1}^{\pinf} C^+_m e^{2\kappa_+ mx}.
\end{equation}
Similarly, we have on a neighbourhood of $\minf$ that for some real numbers $C^-_m, \; m\ge 1$,
\begin{equation} \label{Vl-}
V(x) = \sum_{m=1}^{\pinf} C^-_m e^{2\kappa_- mx}.
\end{equation}

{We now provide some definitions of different equations which will be useful to derive Jost solutions for the potential $V$~\eqref{eq:pot_reissner}. We start by defining the Teukolsky equation~\cite[(2.11)]{Hatsuda}.
\begin{defi}
    The following ordinary differential equation, with $F(r)$ defined in~\eqref{MetricFunctionF} and $\Delta(r)$ defined in~\eqref{Eq:Delta(r)},
    \begin{equation}
    \left[ \frac{\textrm{d}}{\textrm{dr}} \Delta(r) \frac{\textrm{d}}{\textrm{dr}} + \frac{r^2z^2}{F(r)} - \frac{2\Lambda}{3}r^2- l(l+1)\right]R(r) = 0, \quad r \in (r_-,r_+),
    \label{Teukolsky}
\end{equation}
where $\Lambda>0$, $z\in \C$ and $l \in \N_0$ is called scalar Teukolsky equation associated with a Reissner--Nördstrom--de Sitter black hole. 
\end{defi}}
{We introduce now another variable $\zeta$ called Möbius transformation:
\[
\zeta(r) := \frac{(r_+ - r_c)(r-r_-)}{(r_+-r_-)(r-r_c)}.
\]
Through this change of variable, the five points $(r_+,r_-,r_c,r_n,\infty)$ are mapped as $(1,0,\infty,\zeta_r,\zeta_\infty)$ where $\zeta_r := \frac{(r_+ - r_c)(r_n-r_-)}{(r_+-r_-)(r_n-r_c)}$ and $\zeta_\infty := \frac{r_+ - r_c}{r_+-r_-}$.
Considering $r \in (r_-,r_+)$ is thus equivalent to $\zeta \in (0,1)$.
Then, we will give links between some solutions to~\eqref{Teukolsky} and to another equation: the Heun's equation.}
\begin{defi}{\cite[Section 31.3]{NIST:DLMF}}
    Heun's differential equation is the following ordinary differential equation:
    \begin{equation} \label{HeunDef}
    y''(\zeta) + \left( \frac{\gamma}{\zeta} + \frac{\delta}{\zeta-1} + \frac{\varepsilon}{\zeta-a} \right)y'(\zeta) + \frac{\alpha \beta \zeta -q}{\zeta(\zeta-1)(\zeta - a)}y(\zeta)=0,~\zeta \in (0,1),
    \end{equation}
    where the coefficients $\alpha,\beta,\gamma,\delta,\varepsilon \in \C$ satisfy the relation: $\alpha + \beta + 1 = \gamma + \delta + \varepsilon$ and $a \in \C$ satisfy $|a| \geq 1$ and $a \neq 1$.
    We denote by $\Hl(a,q;\alpha,\beta,\gamma,\delta;\zeta)$ the solution of \eqref{HeunDef} with exponent $0$ at $\zeta=0$ and such that $\Hl(a,q;\alpha,\beta,\gamma,\delta;0)=1$.
\end{defi}
Due to~\cite[(2.16)]{Hatsuda}, a solution to Teukolsky equation~\eqref{Teukolsky} $R^-(r,z)$ can be written as 
\begin{align}
& R^-(r,z) = \zeta(r)^{-\frac{iz}{2\kappa_-}}(\zeta(r)-1)^{\frac{iz}{2\kappa_+}}(\zeta(r)-\zeta_r)^{\frac{iz}{2\kappa_n}}(\zeta(r)-\zeta_\infty) \nonumber \\
& \times \Hl\left[\zeta_r,\left(\zeta_r\left(\frac{iz}{\kappa_+}+1\right)+\frac{iz}{\kappa_n}+1\right)\left(\frac{-iz}{\kappa_-}\right)-v(z);1 - \frac{iz}{\kappa_-},1 - \frac{iz}{\kappa_c} - \frac{iz}{\kappa_-},1-\frac{iz}{\kappa_-},1+\frac{iz}{\kappa_+};\zeta(r)\right],
\label{Eq:R-Heun}
\end{align}
where 
\[
v(z):= \frac{r_n}{r_c-r_n} + \frac{3l(l+1)+\Lambda r_-(r_- + r_+)}{\Lambda (r_c - r_n)(r_--r_+)} - \frac{6ir_-r_c z}{\Lambda(r_c-r_n)(r_c-r_-)(r_- - r_+)}.
\]
Similarly, another solution to~\eqref{Teukolsky} $R^+(r,z)$ has the following expression
\begin{align}
& R^+(r,z) = \zeta(r)^{\frac{iz}{\kappa_-}}(\zeta(r)-1)^{\frac{iz}{\kappa_+}}(\zeta(r)-\zeta_r)^{\frac{iz}{\kappa_n}}(\zeta(r)-\zeta_\infty) \nonumber \\
& \times \Hl\left[1-\zeta_r,1 - \frac{iz}{\kappa_c} + v(z); 1,1 - \frac{iz}{\kappa_c},1+\frac{iz}{\kappa_+},1+\frac{iz}{\kappa_-};1-\zeta(r)\right].
\label{Eq:R+Heun}
\end{align}
We finally state the Proposition which gives an explicit expression for the Jost solutions in the case of a potential given by~\eqref{eq:pot_reissner}.
\begin{prop}
\label{RNdSProp}
Let $V$ be a potential associated with a Reissner--Nordström--de Sitter black hole~\eqref{eq:pot_reissner}. Let $r(x)$ be defined from~\eqref{x(r)} 
 and let $R^-$, $R^+$ be solutions to the Teukolsky equation~\eqref{Teukolsky} defined in~\eqref{Eq:R-Heun} and~\eqref{Eq:R+Heun}. Let $z\in \C$ such that $\frac{iz}{\kappa_+} \in \C \backslash (-\N)$ and $\frac{-iz}{\kappa_-} \in \C \backslash (-\N)$. Then the two Jost solutions associated with \eqref{Premiere} are given by the following formulas for $r \in (r_-,r_+)$:
\begin{equation}
\label{RNdSf-}
f^-_0(x(r),z) := \frac{rR^-(r,z)(r_- - r_n)^{-\frac{iz}{2\kappa_n}}(r_- - r_c)^{-\frac{iz}{2\kappa_c}}(r_+ - r_-)^{-\frac{iz}{2\kappa_+}} (r_+-r_c)^{\frac{iz}{2\kappa_-}}}{r_-(-1)^{\frac{iz}{2\kappa_+}}(-\zeta_r)^{\frac{iz}{2\kappa_n}}(-\zeta_{\infty}) (r_- - r_c)^{\frac{iz}{2\kappa_-}} (r_+ - r_-)^{\frac{iz}{2\kappa_-}}}
\end{equation}
and
\begin{equation}
\label{RNdSf+}
f^+_0(x(r),z) := \frac{rR^+(r,z)(r_+ - r_n)^{\frac{iz}{2\kappa_n}}(r_+ - r_c)^{\frac{iz}{2\kappa_c}}(r_+ - r_-)^{\frac{iz}{2\kappa_-}} (r_+-r_c)^{\frac{iz}{2\kappa_+}}(r_+-r_-)^{\frac{iz}{2\kappa_+}}}{r_+(1-\zeta_r)^{\frac{iz}{2\kappa_n}}(1-\zeta_{\infty})(r_c-r_-)^{\frac{iz}{2\kappa_+}}},
\end{equation}
{where $x(r)$ is defined in~\eqref{x(r)}.}
\end{prop}
\begin{proof}
We first link the solutions to the Schrödinger equation~\eqref{Premiere} to solutions to Teukolsky equation~\eqref{Teukolsky}. Then, we link solutions to Teukolsky equation to solutions to Heun's equation~\eqref{HeunDef}. Using asymptotics, we choose a particular Heun's function and finally we check that these expressions are the wanted Jost solutions.

We start from $y(x)$ a solution to the Schrödinger equation~\eqref{Premiere} for $V$ defined in~\eqref{eq:pot_reissner}. Then, we consider $\tilde{y}(r) = y(x(r))$ for $r \in (r_-,r_+)$. So $\tilde{y}'(r) = x'(r) y'(x(r)) = \frac{1}{F(r)} y'(x(r))$ and $\tilde{y}''(r) = \frac{-F'(r)}{F(r)^2}y'(x(r)) + \frac{1}{F(r)^2}y''(x(r))$, and so $y''(x(r))=F(r)^2 \tilde{y}''(r) + F(r)F'(r)\tilde{y}'(r)$. Substituting in~\eqref{Premiere} we obtain that $\tilde{y}$ satisfies
\begin{equation}
\label{eq:ytilde}
    -\tilde{y}''(r)F(r)^2 - F'(r)F(r)\tilde{y}'(r) + \tilde{V}(r)\tilde{y}(r) = z^2\tilde{y}(r), \quad r \in (r_-,r_+),
\end{equation}
where $\tilde{V}(r) := F(r)\frac{l(l+1)}{r^2} + F(r) \frac{F'(r)}{r} + \frac{2}{3}\Lambda F(r)$.

Let us now consider the change of variables $R(r) = \frac{\tilde{y}(r)}{r}$ where $r \in (r_-,r_+)$.
We have that $\tilde{y}'(r) = R(r) + rR'(r)$ and $\tilde{y}''(r) = 2R'(r) + rR''(r)$.
By substituting in~\eqref{eq:ytilde}, we obtain
\[
(2R'(r) + rR''(r))F(r)^2 + F'(r)F(r)(R(r) + rR'(r)) - \tilde{V}(r)rR(r) + z^2rR(r) = 0. 
\]
Using that $\tilde{V}(r)rR(r) = F(r)\left( \frac{l(l+1)}{r} + F'(r) + \frac{2}{3}\Lambda r\right) R(r)$, we have
\[
\begin{split}
& (2R'(r) + rR''(r))F(r)^2 + F'(r)F(r)rR'(r) - F(r)\left( \frac{l(l+1)}{r} + \frac{2}{3}\Lambda r\right) R(r) + z^2rR(r) = 0 \\
\end{split}
\]
Since $0<r_-<r$ and $F>0$ on $(r_-,r_+)$, we have factorising the previous equation by $F(r)/r$
\begin{equation}
\label{eq:AlmostTeukolsky}
 2rF(r)R'(r) + r^2F(r)R''(r) + r^2 F'(r) R'(r) + \left(\frac{r^2z^2}{F(r)} - \frac{2}{3}\Lambda r^2 - l(l+1) \right) R(r) = 0.   
\end{equation}
Using $\Delta(r)$ defined in~\eqref{Eq:Delta(r)}, the following equality holds:
\[
\frac{\textrm{d}}{\textrm{dr}} \left[ \Delta(r) R'(r)\right] = 2rF(r)R'(r)+ r^2F(r)R''(r) + r^2F'(r)R'(r).
\]
So $\frac{\textrm{d}}{\textrm{dr}} \left[ \Delta(r) R'(r)\right]$ corresponds exactly to the first three terms of~\eqref{eq:AlmostTeukolsky}.
Then $R$ is solution to Teukolsky equation~\eqref{Teukolsky}. 
Therefore the Jost solution to equation~\eqref{Premiere} with $V$ defined in~\eqref{eq:pot_reissner} equivalent to $\exp(-izx)$ at $-\infty$  is proportional to $r(x) R^-(r(x),z)$ provided that we can find a multiplicative factor leading to the correct asymptotics.

First we have~\cite[Section 3.1]{Hatsuda}
\[
\begin{split}
\Hl\left[\zeta_r,\left(\zeta_r\left(\frac{iz}{\kappa_+}+1\right)+\frac{iz}{\kappa_n}+1\right)\left(\frac{-iz}{\kappa_-}\right)-v(z);1 - \frac{iz}{\kappa_-},1 - \frac{iz}{\kappa_c} - \frac{iz}{\kappa_-},1-\frac{iz}{\kappa_-},1+\frac{iz}{\kappa_+};\zeta(r)\right] \equ{r \to r_-} 1,
\end{split}
\]
and since $\zeta(r) \rightarrow_{r \to r_-} 0$, we have
\begin{equation} \label{Eq:SimZeta0}
\zeta(r)^{-\frac{iz}{2\kappa_-}}(\zeta(r)-1)^{\frac{iz}{2\kappa_+}}(\zeta(r)-\zeta_r)^{\frac{iz}{2\kappa_n}}(\zeta(r)-\zeta_\infty) \equ{r \to r_-} \zeta(r)^{-\frac{iz}{2\kappa_-}}(-1)^{\frac{iz}{2\kappa_+}}(-\zeta_r)^{\frac{iz}{2\kappa_n}}(-\zeta_\infty),
\end{equation}
so that
\begin{align} \label{Eq:Simy2-}
r R^-(r,z) &\equ{r \to r_-} 
r_- \; \zeta(r)^{-\frac{iz}{2\kappa_-}}(-1)^{\frac{iz}{2\kappa_+}}(-\zeta_r)^{\frac{iz}{2\kappa_n}}(-\zeta_\infty) \\
& \equ{r \to r_-} (r-r_-)^{-\frac{iz}{2\kappa_-}} r_- \left(\frac{(r_+-r_-)(r_--r_c)}{(r_+ - r_c)}\right)^{\frac{iz}{2\kappa_-}}  (-1)^{\frac{iz}{2\kappa_+}}(-\zeta_r)^{\frac{iz}{2\kappa_n}}(-\zeta_\infty).
\end{align}
Second
\[
 e^{-ix(r)z} \equ{r \to r_{-}}  (r-r_-)^{-\frac{iz}{2\kappa_-}}(r_--r_n)^{-\frac{iz}{2\kappa_n}}(r_- - r_c)^{-\frac{iz}{2\kappa_c}}(r_+-r_-)^{-\frac{iz}{2\kappa_+}}.
\]
Hence 
\[
\tilde{f}^-_0(r,z):= \frac{rR^-(r,z)(r_- - r_n)^{-\frac{iz}{2\kappa_n}}(r_- - r_c)^{-\frac{iz}{2\kappa_c}}(r_+ - r_-)^{-\frac{iz}{2\kappa_+}} (r_+-r_c)^{\frac{iz}{2\kappa_-}}}{r_-(-1)^{\frac{iz}{2\kappa_+}}(-\zeta_r)^{\frac{iz}{2\kappa_n}}(-\zeta_{\infty}) (r_- - r_c)^{\frac{iz}{2\kappa_-}} (r_+ - r_-)^{\frac{iz}{2\kappa_-}}}
\equ{r \to r_-} e^{-ix(r)z}.
\]
Using that $\tilde{f}^-_0(r,z) = f^-_0(r(x),z)$,
we obtain $f^-_0(r(x),z)$ defined in~\eqref{RNdSf-} is a Jost solution.
Finally, a similar results for $\tilde{f}_0^+(r,z)$ gives in the end~\eqref{RNdSf+}.
\end{proof}
\begin{rem}[Poles]
    \label{rem:poles3}
Due to the fact that $\Hl(a,q;\alpha,\beta,\gamma,\delta;\zeta)$ has poles for $\gamma \in \C \backslash \N_0$ \cite[31.3(i)]{NIST:DLMF}, $f^-_0$ (respectively $f^+_0$) is not defined for $z \in  -i\kappa_-\N$ (respectively $z=i\kappa_+\N$). With the same reasoning as in Remark~\ref{rem:poles1}, the Wronskian when the reference potential is given by~\eqref{eq:pot_reissner} has possible poles in $(i\kappa_+\N)\cup(-i\kappa_-\N)$.
\end{rem}

Numerically, we rather work in this case with the variable $r$, and so with Jost solutions $\tilde{f}^\pm_0$, noting that the quantity $r(x)$ has no closed-form formula, so that working in variable $x$ could introduce numerical errors and be time-consuming. For this reason, and denoting $r_\alpha = r(\alpha)$ and $r_\beta = r(\beta)$ the Cauchy problem~\eqref{CP2} is replaced {by} the slightly different problem
\[
\left\{
\begin{array}{l}
    -\tilde{y}''(r) - \frac{F'(r)}{F(r)}\tilde{y}'(r) + \frac{\tilde{V}(r)}{F(r)^2}\tilde{y}(r) = \frac{z^2}{F(r)^2}\tilde{y}(r), \quad r \in (r_\alpha,r_\beta), \\
\tilde{y}(r_\alpha) = \tilde{f}^-_0(r_\alpha),z)\\
\tilde{y}'(r_\alpha) = (\tilde{f}^-_0)'(r_\alpha),z).
\end{array}
\right.
\]
Then, the Wronskian $\tilde{w}(r,z) := [\tilde{f}^-_0(r,z),\tilde{f}_0^+(r,z)]$ depends on the variable $r$. Due to the relation between $f^\pm_0$ and $\tilde{f}^\pm_0$, we use the next formula to compute the wanted Wronskian $w(z) = [f^-_0(x,z),f^+_0(x,z)]$. Note that 
\[
\tilde{w}(r,z) = x'(r)[f^-_0(x,z),f^+_0(x,z)] = \frac{1}{F(r)}w(z)
\]
and so
\[
w(z) = \tilde{w}(r,z)F(r).
\]

\section{Numerical Results}
\label{Sec3}

In this section, we present numerical results showing how the deflated Newton's method combined with well-chosen reference potentials provides accurate resonances. In Subsection~\ref{Benchmark}, we 
show how the method behaves when modifying discretization parameters and reference potentials.
In the subsequent sections, we study the resonances of perturbations of three types of reference potentials: Poschl--Teller potentials, exponentially decaying potentials, and Reissner--Nordström--de Sitter black holes.
The Julia code for reproducing the figures can be found at~\cite{Zenodo}. 
To give a few computational details, the resolution of Cauchy problems is performed with the package DifferentialEquations~\cite{DiffEq}. 
In Algorithm~\ref{Algo1}, the derivative of $M$ is obtained with automatic differentiation using ForwardDiff~\cite{ForwardDiff}. 
Due to a compatibility issues between the packages DifferentialEquations and ForwardDiff,
the function $w'$ used in the Newton's deflated method is computed with a fourth-order centered finite difference scheme. 
The contour integral is computed numerically using QuadGK~\cite{QuadGK}. The hypergeometric and Bessel functions are available in SpecialFunctions.
Finally Heun's functions were not available in Julia, therefore we translated to Julia the MATLAB code developed by Motygin~\cite{Motygin} for these simulations.
Finally empirical considerations led us to choose a maximum of $200$ iterations for each Newton algorithm.

\subsection{Benchmark of the method}
\label{Benchmark}

\subsubsection{Varying resolution parameters}

Unless explicitly stated otherwise, the reference potential $V_0$ considered in this section is the Pöschl--Teller potential with $\lambda= 1$. The interval $[\alpha,\beta]$ is typically taken as $[-1,5]$, and we locate resonances within a rectangle similar to Figure~\ref{fig:Rectangle}. 
First we observed numerically that changing the 
interval $[\alpha,\beta]$ to solve the Cauchy problem~\eqref{CP1} without changing $q$ only changes the values of the resonances by at most $10^{-9}$ from $[-1,5]$ to $[-6,10]$, for a perturbation with support in $[-1,5]$.

Second the numerical method chosen to compute the Jost solution does not impact the results.
 We indeed computed the supremum of the difference between the reference case given by Verner’s 9th-order method that we used in the following test cases, and other methods 
among which
Verner’s methods of lower order, the Tsitouras method, the Dormand–Prince methods of different orders, and the classical Runge–Kutta method. We observed a maximum difference on the resonances of $10^{-12}$ which is indeed negligible.

Third we discuss the impact of the numerical precision of the Jost solutions on the accuracy of resonances. Recall that one of the Jost solutions is obtained by solving a Cauchy problem, after which we compute the Wronskian and determine its zeros. 
The considered perturbation is $q = \indic{[-1,5]}$.
We observe in Figure~\ref{fig:PreciJostSol} that for a tolerance on the Jost solutions smaller than $10^{-9}$, the supremum of the distance between the resonances is about $10^{-11}$.
We therefore chose relative precision parameter of $10^{-10}$ for the Jost solutions in the following.

\begin{figure}[htb!]
    \centering
    \includegraphics[width = 8cm]{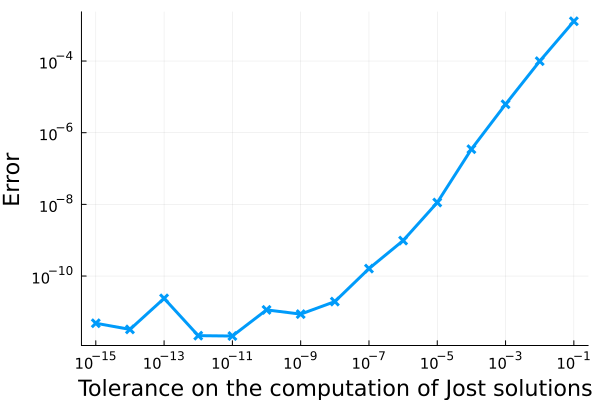}
    \captionsetup{justification=centering}
    \caption{Supremum of the difference between resonances obtained for varying tolerances on Jost solutions and a reference solution computed with a tolerance of $10^{-16}$ on the Jost solutions.}
    \label{fig:PreciJostSol}
\end{figure}

\subsubsection{Importance of the choice of the reference potential}
\label{ImpChoicePot}

We now compare two approaches which should naively give the same resonances in the limit of large support $[-L,L]$, with $L \to \pinf$. The first one consists in considering $V_0$ as a potential and no perturbation (i.e. $q=0$).  In the second one, we start with the null reference potential and consider a compactly supported perturbation of this null potential, that is $q = V_0 \indic{[-L,L]}$. For this example, we took $V_0 = \frac{1}{\cosh^2(\cdot)}$ a Pöschl–Teller potential, since in this case the location of resonances is theoretically known.
According to \cite[Theorem 4]{VA}, in the case $\lambda = 0$, or~\cite[Theorem 6]{Zworski87} resonances $z_{\pm j}$ have the following asymptotics:
\[
    \begin{split}
    z_{\pm j} & = \pm \frac{\pi}{4L} \left( 2j + 2 \pm 1 \right)  - \frac{i}{L} \log \left( \frac{j \pi}{2L}\right) + \frac{i}{4L} \log(\cosh^{-4}(L)) +o(1),
    \end{split}
\]
    as $j \in \N$ tends to infinity and $o(1)$ stands for $j \to \pinf$.
\begin{figure}[htb!]
    \centering
    \includegraphics[width = 8cm]{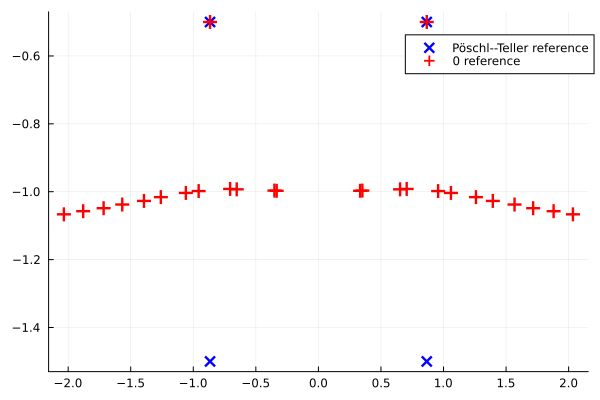}
    \captionsetup{justification=centering}
    \caption{Resonances for different reference potentials.}
    \label{ImpRefPot}
\end{figure}

For $L$ large enough (for Figure \ref{ImpRefPot} we chose $L = 10$), we have 
\[
\frac{i}{4L} \log(\cosh^{-4}(L))= \frac{i}{4L}\log(e^{-4L})(1+o(1)) = -i(1+o(1)).
\]
So the presence of the quantity $\frac{i}{4L} \log(\cosh^{-4}(L))$ explains that, for a potential which decays more slowly than super-exponentially such as Pöschl--Teller potential, we observe in Figure \ref{ImpRefPot} the appearance of preasymptotic lines of resonances instead of logarithmic branches.
We emphasize here that this preasymptotic lines are independent of the numerical method chosen to compute the resonances, including those presented in \cite{BZ} or \cite{Levitt}, but depend on the truncation of the potential. 

This phenomenon is a clear motivation for our choice to start with a carefully chosen reference potential $V_0 \neq 0$
and to consider a compactly supported perturbation $q$, rather than immediately considering a perturbation of the free potential in the form $(V_0 + q)\indic{[-L,L]}$ for $L$ large enough. 

\subsection{Location of resonances (Pöschl--Teller potentials)}
\label{LocRes}

In this section, we compute and analyze some phenomena related to resonances of perturbations of 
Pöschl--Teller potentials, and compare them to existing theoretical results.
Unless otherwise stated, the coefficient for the Pöschl--Teller potential is $\lambda=1$.

\subsubsection{Different compactly supported perturbations}
\label{DiffCompactSuppPerturb}

To start with we study how different are the resonances of perturbations of the same Pöschl--Teller reference potential, with perturbations having the same support and same values at the boundary of this support. 
In theory~\cite[Theorem 4]{VA}, the resonances should asymptotically lie on two logarithmic branches
independent of 
the values of the perturbation within the support, and only depending on the values of the perturbation at the border of its support.
To check this we choose two perturbation potentials $q_1$ and $q_2$ defined by 
\begin{equation} \label{Indic}
q_1(x) = \left\{
    \begin{array}{ll}
         -1 & \mbox{if } -\pi\leq x \leq {\pi}\\
       0 & \mbox{elsewhere}
    \end{array}
\right., \quad
\text{and}
\quad
q_2(x) = \left\{
    \begin{array}{ll}
         \cos(x) & \mbox{if } -\pi\leq x \leq \pi\\
       0 & \mbox{elsewhere}.
    \end{array}
\right.
\end{equation}
\begin{figure}[htb!]
    \centering
    \includegraphics[width = 8cm]{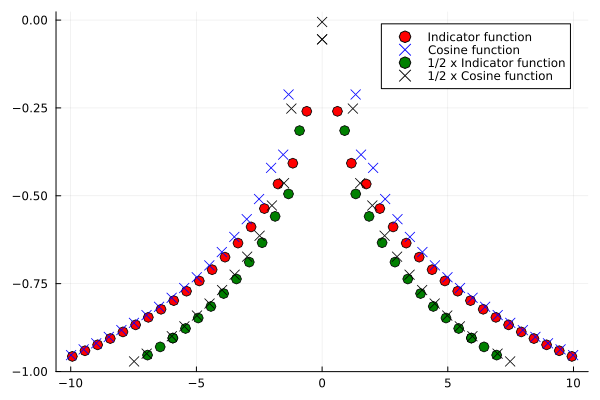}
    \captionsetup{justification=centering}
    \caption{Location of resonances for different perturbations of the same Pöschl--Teller potential.}
    \label{LocResDiffPot}
\end{figure}
In Figure \ref{LocResDiffPot}, we observe that the two logarithmic lines of resonances for $q_1$ and $q_2$ (respectively $1/2q_1$ and $1/2q_2$) seem to merge as the modulus of the resonances increases. This is consistent with the results presented in \cite[Theorem 4]{VA},
 where the logarithmic lines do not depend on the value of the perturbation within the support but only at the jump discontinuities.
We also observe the appearance of resonances close to the real axis in the case $1/2q_2$, which do not seem to correspond to any resonances in the case $q_2$. This phenomenon is due to the fact that, for the perturbation $q_2$, the potential presents a small well. This leads to the formation of bound states, i.e. eigenvalues. When we then consider $\tfrac12 q_2$ instead of $q_2$, the depth of this well is reduced, which decreases the capacity of the potential to create bound states and so transforms the former eigenvalues into resonances. This phenomenon is e.g. shown in~\cite[Figure 3]{Belchev}.

\subsubsection{Stability of small resonances}
\label{StabInstab}

We then study the stability of small resonances, i.e. resonances closer to the imaginary axis.
In particular, we observe how the resonances of an unperturbed  Pöschl--Teller potential are modified by adding a real-valued, integrable and compactly supported potential $q$ of varying amplitude.
Namely, we take $q = \indic{[-1,5]}$ and we consider potentials for $\tau \in [0,1]$
\begin{equation}
\label{eq:Vtau1}
    V_\tau^1: x\in \R\mapsto V_\tau^1(x) := \frac{\lambda}{\cosh^2(x)} + \tau q(x).
\end{equation}
\begin{figure}[htb!]
    \centering
    \begin{subfigure}[t]{0.48\textwidth}
        \includegraphics[width=\linewidth]{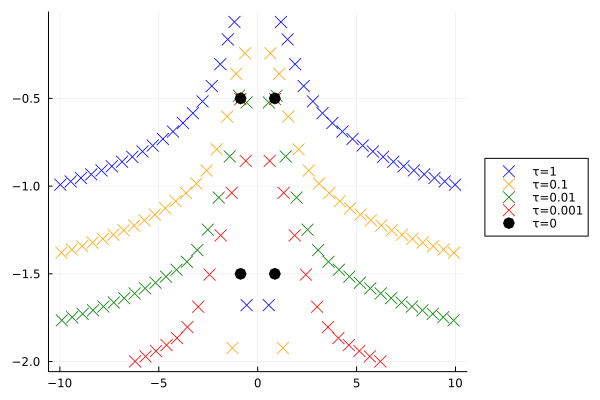}
        \captionsetup{justification=centering}
        \caption{Resonances for different values of the parameter $\tau$ for the potential $V^1_{\tau}$
        }
        \label{InstRes}
    \end{subfigure}
    \hfill
    \begin{subfigure}[t]{0.48\textwidth}
        \includegraphics[width=\linewidth]{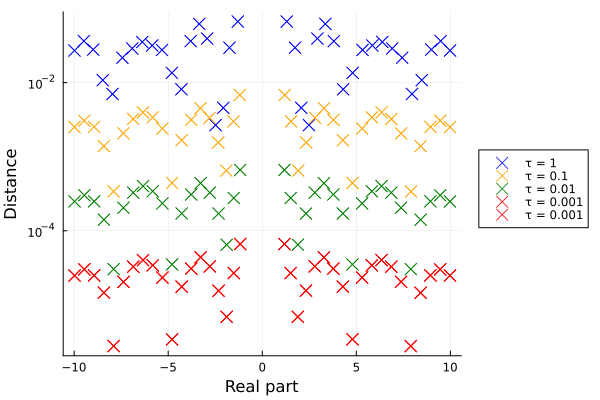}
        \captionsetup{justification=centering}
        \caption{Distance between resonances obtained for the reference case $V_0^2$ and resonances obtained for any $\tau$}
        \label{ImagPart}
    \end{subfigure}
    \caption{Resonances for different values of $\tau \in [0,1]$ and with two reference potentials $V^1_\tau$ and $V^2_\tau$.}
\end{figure}
We plot in Figure~\ref{InstRes} the resonances for several values of $\tau$ on the domain $\{ z = x - iy ~|~ x \in [-10,10] ~\text{and}~ y \in (0,2) \}$. Since Pöschl--Teller resonances are located 
on two lines parallel to the imaginary axis, and perturbations thereof on two logarithmic lines~\cite[Theorem 4]{VA}, the large resonances are necessarily unstable. However, we notice that, 
the closer the resonances are to the imaginary axis, the more stable they are. This confirms the results presented in~\cite[Figure 7]{JLJ}, obtained with a pseudospectrum method. 
When the parameter $\tau$ is smaller than $10^{-7}$ we observe that there are only four resonances in the considered domain. We compute these four resonances for different values of $\tau \in [0,10^{-7}]$ and then plot the maximum error with respect to the exact ones for $\tau = 0$ in Figure~\ref{StabTau} (in blue). 
The error in log-log scale follows a straight line of slope $1.00$ 
hence
the distances between two corresponding resonances are proportional to the perturbation parameter $\tau$, indicating the stability of the resonances. 

We then study the stability of resonances for a perturbation of an already perturbed Pöschl--Teller potential that is we take $q_1 = \indic{[-1,5]}$ and $q_2 = \indic{[-0.5,2.5]}$, and consider, for $\tau\in[0,1]$,
\begin{equation}
\label{eq:Vtau2}
    V_\tau^2: x\in \R \mapsto V_\tau^2(x) := \frac{\lambda}{\cosh^2(x)} + q_1 + \tau q_2(x),
\end{equation}
Since all these resonances share almost the same real part, we plot the distance between the imaginary parts of the resonances for different values of $\tau$ in Figure~\ref{ImagPart}.
The error on the resonances systematically decreases and seems of the order of the perturbation, suggesting stability of all resonances.
In order to gain a better understanding of Figure~\ref{ImagPart}, especially what happens for smaller values of $\tau$, we also plot in Figure~\ref{StabTau} the supremum of the distance between resonances for $\tau = 0$ and any $\tau \in [0,1]$ (in red). As before, the 
points for $\tau \geq 10^{-12}$ seem to lie on a straight line with slope $1.00$ indicating that the differences between two resonances are proportional to the parameter $\tau$. As expected, we observe that the distances are way lower for the second potential~\eqref{eq:Vtau2} than for the first one~\eqref{eq:Vtau1}.
\begin{figure}[H]
    \centering
    \includegraphics[width = 8cm]{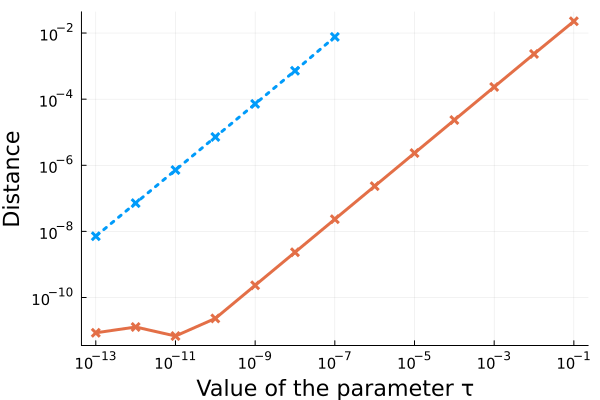}
    \captionsetup{justification=centering}
    \caption{Supremum of the distance between resonances obtained for potentials $V^1_{\tau}$ (dotted line) and potentials $V^2_{\tau}$ (continuous line) for different perturbation parameters $\tau$.
    }
    \label{StabTau}
\end{figure}
\subsubsection{Increasing values of the Pöschl--Teller parameter}
\label{SectionLambda}

Let us recall that our original motivation for the study of~\eqref{Premiere} was to understand the behavior of small resonances for black holes. 
In a three-dimensional infinite hyperbolic cylinder (see~\cite[Section 2]{VA} for more details), and assuming that the considered perturbation is radial, after a suitable change of variable, the problem of interest is transformed into a sequence of one-dimensional problems involving Pöschl--Teller potentials of the form
\begin{equation}
    -y''(x) + \left( \frac{\lambda_k}{\cosh^2(x)} + q(x) \right) y(x) = z^2 y(x), \quad x\in \R,
    \label{Eq:Sch3d}
\end{equation}
with $\lambda_k := k(k+1)$, for $k\in\N_0$, and the resonances for the three-dimensional problem corresponds to the union of the resonances generated by~\eqref{Eq:Sch3d} for $k \in \N_0$.
\begin{figure}[htb!]
    \centering
    \includegraphics[width = 8.5cm]{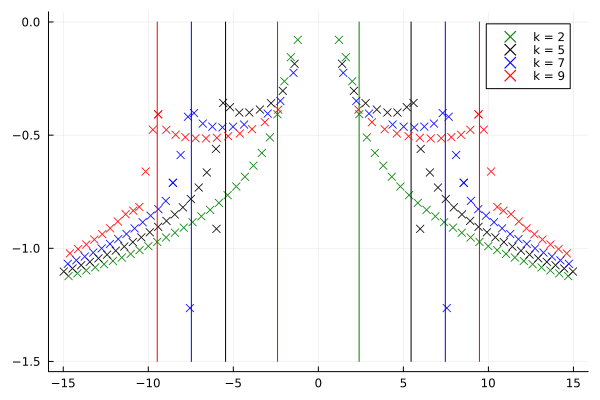}
    \captionsetup{justification=centering}
    \caption{Location of resonances for different values of the parameter $k$.}
    \label{DiffValuek}
\end{figure}

In Figure \ref{DiffValuek}, we observe a jump in the resonances which are at first not located on logarithmic branches, and slightly increase in imaginary part, 
before all diving on the two logarithmic branches
independent from the parameter $\lambda_k$ 
as predicted theoretically~\cite{VA}.
An interpretation of the appearance of these preasymptotic behavior lines is the fact that the Pöschl–Teller potential is more influential than the perturbation in the behavior of small resonances. Furthermore, this phenomenon becomes more pronounced as the parameter $k$ increases.
{The abscissas where the resonances join logarithmic branches, denoted with colored vertical lines in Figure~\ref{DiffValuek}, correspond to values $\pm \sqrt{\lambda_k - \frac{1}{4}}$, \textit{i.e.} the real part of resonances for an unperturbed Pöschl--Teller potential with a constant $\lambda_k$~\cite[Figure 1]{VA}.}
\subsubsection{Comparison with exactly exponentially decaying reference potentials}
We now compare the resonances of a perturbed Pöschl--Teller potential and a perturbed exactly exponentially decaying potentials as defined in~\eqref{kappa+} and~\eqref{kappa-}, {where the two perturbations are $q = \indic{[-1,5]}$}.
To do so, we approximate the Pöschl--Teller potential ($\lambda=1$) by an exponentially decaying potential in the following manner: we 
choose $R_-=-1$, $R_+=5$, we use the Pöschl--Teller potential in $[R_-,R_+]$, and we then choose $\kappa_-=-1,\kappa_+ =1$ following the asymptotic behavior of the Pöschl--Teller potential and choose $\mu-=\mu_+=4$ so that the corresponding potential is continuous.

As expected we observe two logarithmic branches of the resonances which seem to be the same. Indeed, we plot in Figure~\ref{fig:Fig9} the difference between resonances as a function of the real part of the resonances, and observe that the difference goes to zero.
Thus Pöschl–Teller potentials can be regarded as a good (and simple) approximation of exactly exponentially decaying potentials, especially when $\mu_-=\mu_+$.
\begin{figure}[H]
    \centering
    \includegraphics[width = 8cm]{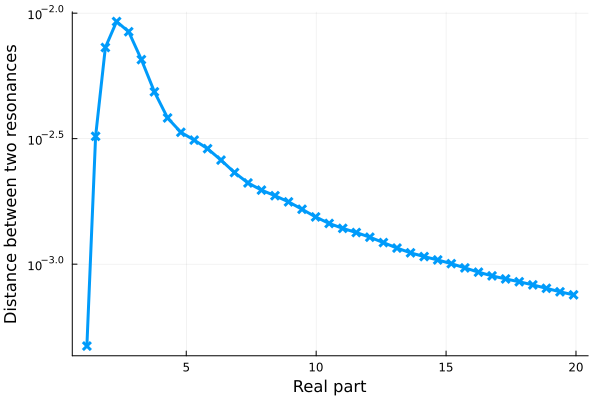}
    \captionsetup{justification=centering}
    \caption{Distance between resonances for exactly exponentially decaying and Pöschl--Teller potentials.}
    \label{fig:Fig9}
\end{figure}

\subsection{Resonances for Reissner--Nordström--de Sitter black holes}

We now turn to the computation of resonances for Reissner--Nordström--de Sitter black holes and perturbation thereof.

\subsubsection{Unperturbed black holes}
\begin{figure}[htb!]
    \centering
    \includegraphics[width = 13cm]{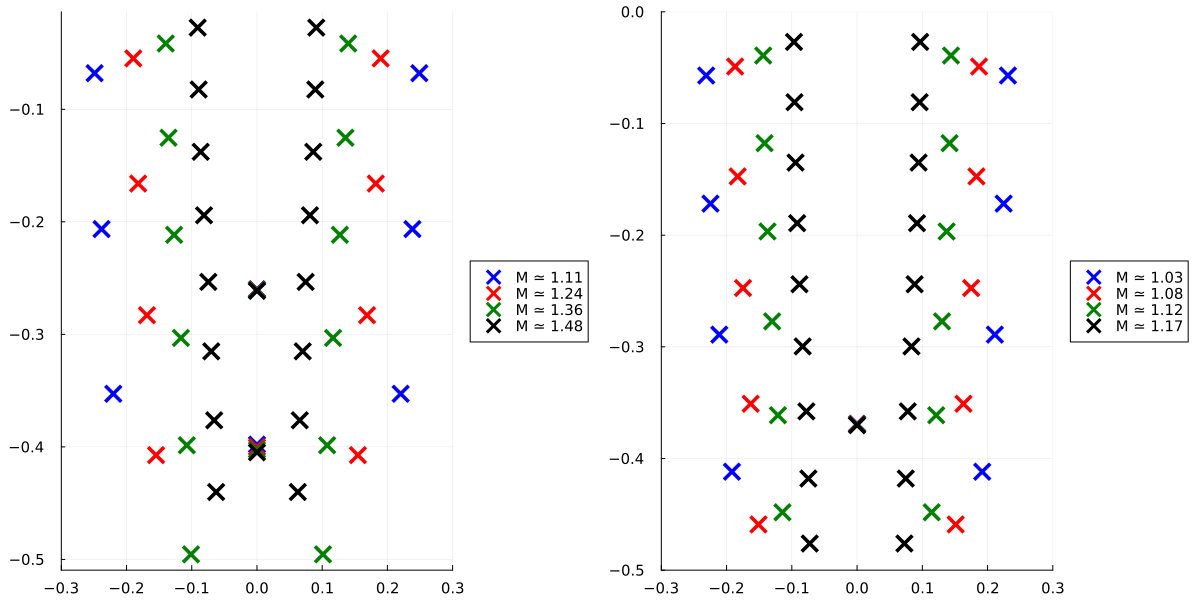}
    \captionsetup{justification=centering}
    \caption{Location of resonances for different values of the mass of a Reissner--Nordström--de Sitter black hole when $\Lambda = 0.05$ (left) and $\Lambda = 0.1$ (right).
    }
    \label{UnperRNdS}
\end{figure}
We start with the unperturbed case.
In Figure~\ref{UnperRNdS}, we plot the resonances for different values of the mass parameter. We fix the values $Q = 1$, $\Lambda \in \{0.01,0.5\}$, $l=1$, and then determine the admissible interval $[M_1, M_2]$ for the mass parameter~\eqref{Mokdad}. 
We observe that the closer the mass is to $M_2$, the larger the number of resonances and the more parallel the lines along which they are located.
This behavior seems similar to the one pointed out in or \cite[Figure 4,5,6]{Destounis}, \cite[Figure 1.1]{HinXie} or \cite[Figure 1]{HitZwo}.

\subsubsection{Comparison to exactly exponentially decaying potentials}

We now compare the location of resonances between exactly exponentially decaying potentials and Reissner--Nordström--de Sitter black hole derived potentials. Indeed, as mentioned in the introduction of Subsection~\ref{Eedp}, exactly exponentially decaying potentials can be seen as the dominant term in the development of the potential on a neighbourhood of $\pm \infty$ for a Reissner--Nordström--de Sitter black hole. For this comparison, we perturb both reference potentials with $q =\indic{[-1,4]}$. We take $(\mu_+,\mu_-) = (C_1^+,C_1^-)$ defined in \eqref{Vl+} and \eqref{Vl-} for the exactly exponentially decaying potential.
We also use the values of surface gravities $\kappa_\pm$ defined by the parameter of the black hole $(M,Q,\Lambda,l)=(0.01,0.01,5,1)$ or $(M,Q,\Lambda,l)=(0.1,0.1,0.1,1)$. We observe in Figure~\ref{BHRNdS} that the logarithmic branches on which the resonances are located are clearly different. {Contrarily, we observe in Figure~\ref{BHRNdS2} that the logarithmic branches seem to merge. A possible explanation for the difference in behavior between the two plots is the value of $C^\pm_1$. When we are in the situation of Figure~\ref{BHRNdS}, we note that $C^+_1 \approx 18.25$ and $C^-_1 \approx 13.02$. In Figure~\ref{BHRNdS2}, we have $C^+_1 \approx 0.47$ and $C^-_1 \approx 0.29$. Furthermore, the coefficients introduced in the expansion in power series of $V$~\eqref{Vl+} and~\eqref{Vl-} satisfy $|C_m| \leq e^{-\nu m}$ for $\nu$ a positive constant \cite[Proposition A.1]{Kehle}. Thus, starting with $C^\pm_1$ values that are close to zero implies that the subsequent coefficients become negligible more quickly, and so the error made by using only the first term of the series~\eqref{Vl+} and~\eqref{Vl-} is smaller. Therefore, it seems that for a choice of parameters $M$, $Q$, $\Lambda$, and $l$ such that $C^\pm_1$ are quite small, exactly exponentially decaying potentials provide a good approximation of the potential arising from Reissner--Nordström--de Sitter black holes.}
\begin{figure}[htb!]
    \centering
    \begin{subfigure}[t]{0.48\textwidth}
        \includegraphics[width=\linewidth]{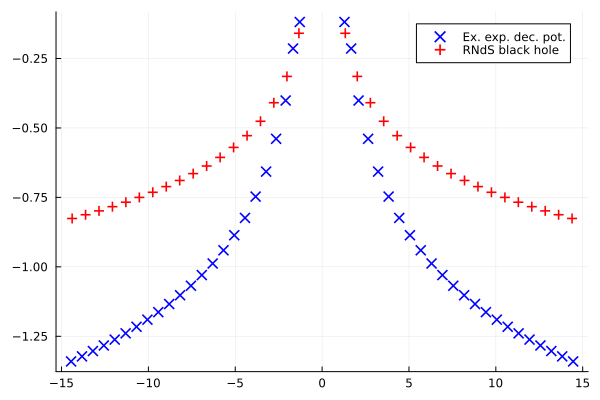}
        \captionsetup{justification=centering}
        \caption{$(M,Q,\Lambda,l)=(0.01,0.01,5,1)$}
        \label{BHRNdS}
    \end{subfigure}
    \hfill
    \begin{subfigure}[t]{0.48\textwidth}
        \includegraphics[width=\linewidth]{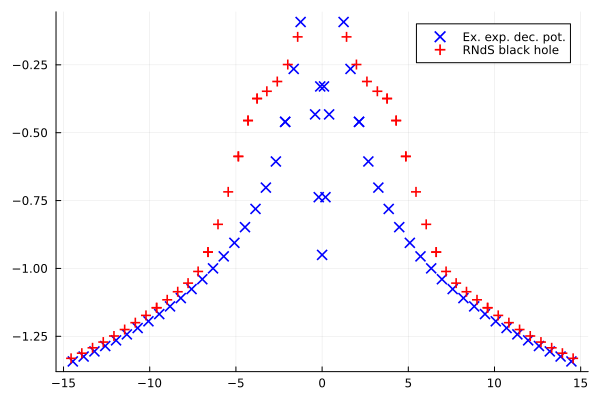}
        \captionsetup{justification=centering}
        \caption{$(M,Q,\Lambda,l)=(0.1,0.1,0.1,1)$}
        \label{BHRNdS2}
    \end{subfigure}
    \caption{Location of resonances for exactly exponentially decaying potential and Reissner--Nordström--de Sitter's potential {perturbed by $q = \indic{[-1,4]}$.}}
\end{figure}

\subsubsection{Different compactly supported perturbations of a Reissner--Nordström--de Sitter black hole}

We now compare the impact of the perturbation on the resonances.
To do so, 
we consider two different perturbations having the same support, the same regularity and the same values at the boundary of this support: $q_1 = \indic{[-\pi/2,3\pi/2]}$ and $q_2 = -\sin \indic{[-\pi/2,3\pi/2]}$.

\begin{figure}[htb!]
    \centering
    \begin{subfigure}[t]{0.48\textwidth}
        \includegraphics[width = 7.5cm]{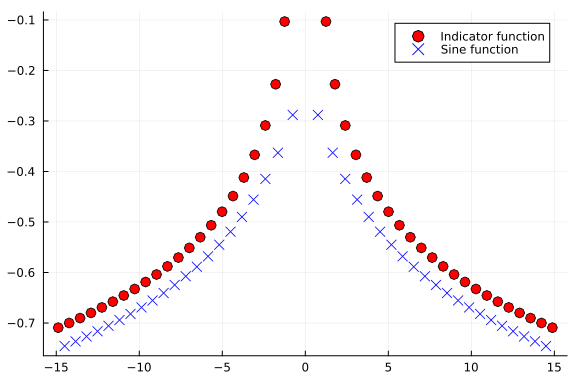}
        \captionsetup{justification=centering}
        \caption{Location of resonances for different type of perturbations of the same Reissner--Nordström--de Sitter black hole.}
        \label{DiffPerturbRNdS}
    \end{subfigure}
    \hfill
    \begin{subfigure}[t]{0.48\textwidth}
            \includegraphics[width = 7.5cm]{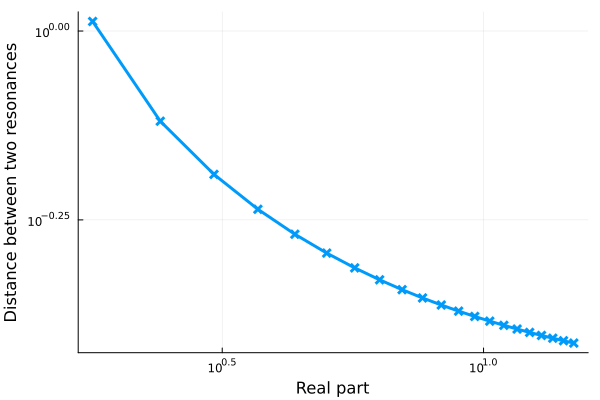}
        \captionsetup{justification=centering}
        \caption{Distance between two resonances for different type of perturbations of the same Reissner--Nordström--de Sitter black hole.}
        \label{DistanceResoRNdS}
    \end{subfigure}
    \caption{Location of resonances for compactly supported perturbed Reissner--Nordström--de Sitter black holes.}
\end{figure}
We observe in Figure~\ref{DiffPerturbRNdS} that the logarithmic lines obtained for the two perturbations are not merging as fast as they did in Figure~\ref{LocResDiffPot}. To check if these two lines will merge at one point, we plot the distance between two resonances.
In Figure~\ref{DistanceResoRNdS}, computing the distance between two resonances of a same pair, we observe that this distance seems to slowly tend to zero.
Actually, the distance between two resonances appears to be proportional to the inverse of the logarithm of the real part.
This leads us to think that, similarly to the case of Pöschl--Teller potentials {described in Figure~\ref{LocResDiffPot}}, the behavior of the perturbation within its support has no impact on the resonances' asymptotics.
The only difference is that in the case of black holes, the distance between resonances for different type of perturbations decreases more slowly than in the case of a Pöschl--Teller potential.

\subsubsection{Stability of small resonances}
This section is the analogue of the Reissner--Nordström--de Sitter case of Section~\ref{StabInstab} where we study the stability of resonances {when the reference potential is a Pöschl--Teller potential}. We consider such black hole with parameters $(M,Q,\Lambda,l)=(1,1,0.1,1)$, which are in the range of the values presented in \cite[Table 1]{Hatsuda}, and a perturbation $q = \indic{[-1,4]}$. We compute the resonances associated with the potential $V_0 + \tau q$ for $\tau \in [0,1]$.
\begin{figure}[htb!]
    \centering
    \includegraphics[width = 8cm]{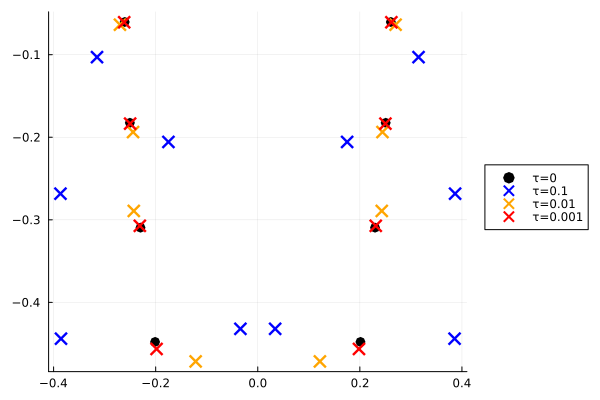}
    \captionsetup{justification=centering}
    \caption{Location of resonances for different perturbations of the same Reissner--Nordström--de Sitter black hole.}
    \label{StabRNdS}
\end{figure}
In Figure~\ref{StabRNdS}, we plot the resonances for different values of $\tau$, and in Figure~\ref{StabtauRNdS} (blue line),
we observe that similarly to the case of a Pöschl--Teller potential, the small resonances are stable.
We are also interested in the behavior of resonances for perturbations of a compactly supported perturbed Reissner--Nordström--de Sitter black hole, \textit{i.e.} resonances obtained for a potential $V_0 + q_1 + \tau q_2$ for $\tau \in [0,1]$. We consider $q_1 = \indic{[-1,4]}$ and $q_2 = \indic{[-1/2,3]}$. 
Similarly as in the Pöschl--Teller case, we observe in Figure~\ref{StabtauRNdS} (orange line) that the resonances are stable with a smaller prefactor than with the perturbation of the reference potential.
\begin{figure}[H]
    \centering
    \includegraphics[width = 8cm]{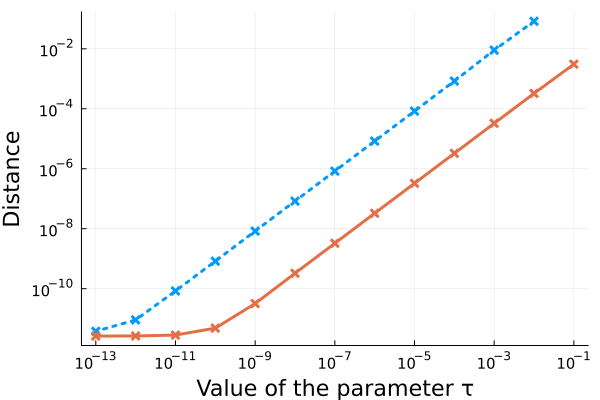}
    \captionsetup{justification=centering}
    \caption{Supremum of the distance between resonances obtained for a Reissner--Nordström--de Sitter black hole (dotted line) and a compactly supported perturbed Reissner--Nordström--de Sitter black hole (continuous line) between the exact results and resonances obtained for any $\tau$.}
    \label{StabtauRNdS}
\end{figure}

\subsubsection{Strong cosmic censorship hypothesis}

As an application of the numerical computation of resonances, we present results related to the  strong cosmic censorship, which is an important conjecture in general relativity, first formulated by Penrose in 1969~\cite{Penrose}. It states that general relativity is a deterministic theory, i.e. that the Cauchy horizon $r_c$ is generically unstable and cannot be crossed without encountering a breakdown of determinism~\cite[Introduction]{Cardoso}.
Mathematically speaking, the strong cosmic censorship hypothesis holds if there exists a resonance $z\in\C$ which satisfies~\cite[(2.26)]{DaveyDiasSola}
\begin{equation}
\label{SCC}
-\frac{\Im(z)}{|\kappa_c|} \leq \frac{1}{2},
\end{equation}
where $\kappa_c$ is defined in~\eqref{kappaeq}.
Resonances of a Reissner--Nordström--de Sitter black holes with parameters $(M,Q,\Lambda)$ are defined as the union over $l \in \mathbb{N}$ of all the resonances of~\eqref{Premiere} with a potential $V$ defined as in \eqref{eq:pot_reissner}.
Even though there exist theoretical and numerical results concerning strong cosmic censorship for various values of black hole parameters and also for different kind of black holes, see e.g.~\cite{Cardoso,DaveyDiasSola,Dias}, we aim to investigate how this conjecture behaves for compactly supported perturbed Reissner--Nordström--de Sitter black hole potentials. 
In practice we check whether~\eqref{SCC} holds for $l\in \{ 0,1,\ldots, 25\}$, and note that in the case where~\eqref{SCC} is not satisfied we fail to prove that the strong cosmic censorship hypothesis holds, rather than prove that the conjecture is violated.

We consider two different perturbations: 
$q_1 = \indic{[-\pi/2,3\pi/2]}$ and
$q_2 = \sin\indic{[-\pi/2,3\pi/2]}$. 
We consider two different charge values $Q$ ($Q=0.1$ and $Q=0.7$), and then using~\eqref{Mokdad}, we determine the admissible values for $\Lambda$ and $M$. We then indicate on Figure~\ref{HeatMap01} (resp. Figure~\ref{HeatMap07}) the parameters for which the hypothesis is valid for $Q=0.1$ (resp. $Q=0.7$).

\begin{figure}[H]
    \centering
    \includegraphics[width = 15cm, height=6.5cm]{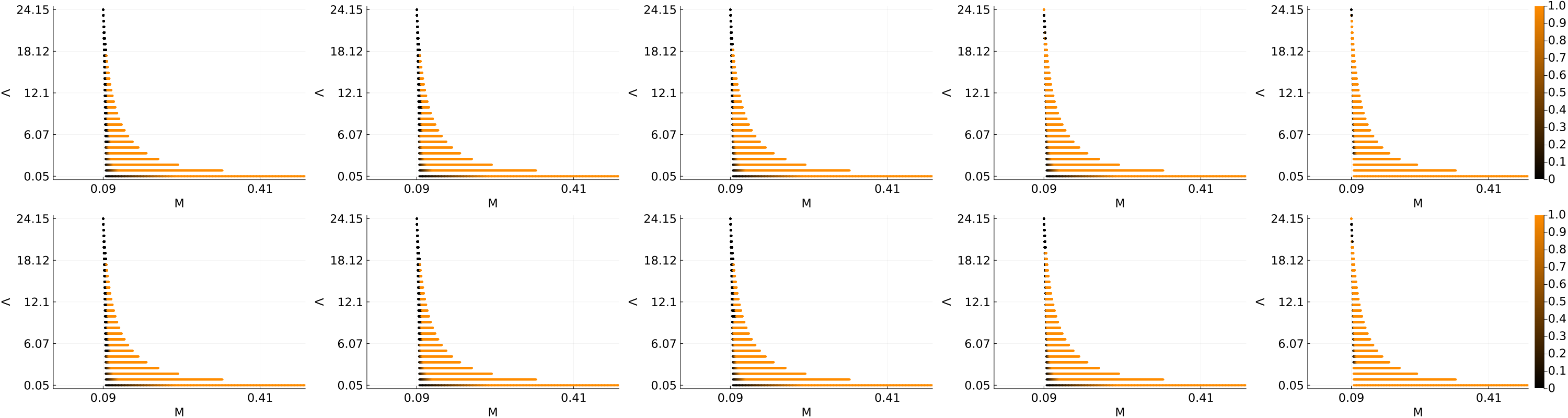}
    \captionsetup{justification=centering}
    \caption{Validation of the strong cosmic censorship hypothesis ($1$ in orange when the hypothesis is satisfied, and $0$ in black otherwise) for a Reissner--Nordström--de Sitter black hole with $Q=0.1$, perturbed by $\tau q_1$ (top) and $\tau q_2$ (bottom) for $\tau \in \{0,10^{-3},10^{-2},10^{-1},1\}$.}
    \label{HeatMap01}
\end{figure}
\begin{figure}[H]
    \centering
    \includegraphics[width = 15cm, height=6.5cm]{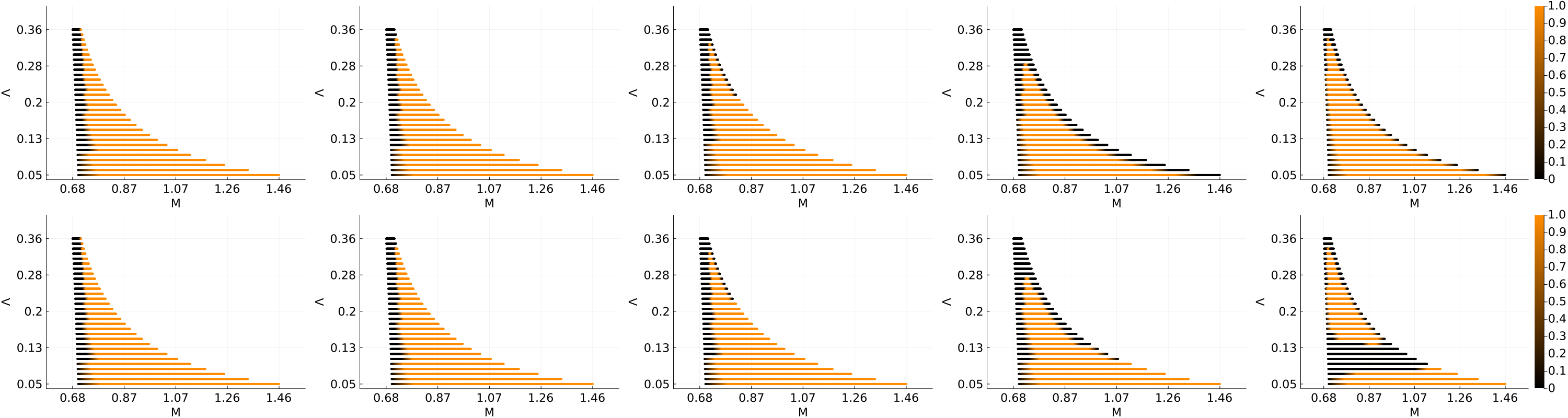}
    \captionsetup{justification=centering}
    \caption{Validation of the strong cosmic censorship hypothesis ($1$ in orange when the hypothesis is satisfied, and $0$ in black otherwise) for a Reissner--Nordström--de Sitter black hole with $Q=0.7$, perturbed by $\tau q_1$ (above) and $\tau q_2$ (below) for $\tau \in \{0,10^{-3},10^{-2},10^{-1},1\}$.}
    \label{HeatMap07}
\end{figure}

Comparing Figure~\ref{HeatMap01} and Figure~\ref{HeatMap07}, we observe that the validity of the strong cosmic censorship hypothesis for black holes with identical perturbations strongly depends on the parameter values, in particular the charge~$Q$.
A common feature of Figures~\ref{HeatMap01} and~\ref{HeatMap07} is that, for the unperturbed cases, we observe that the parameter values for which the strong cosmic censorship hypothesis is not guaranteed to hold correspond to small masses.
Also, as the scaling parameter $\tau$ in the perturbation tends to zero, the hypothesis validation seems to converge to the unperturbed $\tau = 0$ case, as
expected from the stability of the first resonances mentioned in Figure~\ref{StabtauRNdS}.

Finally we provide in Table~\ref{tab} the values of the parameter $l$ for which~\eqref{SCC} is satisfied, illustrating that different behaviors may occur depending on the parameters. Starting from a black hole that does not satisfy the strong cosmic censorship hypothesis (as in Table~\ref{tab}, columns 2-3), one may obtain a perturbed black hole that either satisfies or does not satisfy the hypothesis. Similarly, starting from a black hole that satisfies the strong cosmic censorship hypothesis (as in Table~\ref{tab}, columns 4-5), one may also obtain a perturbed black hole that either satisfies or does not satisfy the hypothesis.
\begin{table}[htb!]
\centering       
\scalebox{0.85}{
\begin{tabular}{|c|c|c|c|c|}
\hline
\multirow{2}{*}{Perturbation} & \multicolumn{4}{c|}{$(M,Q,\Lambda)$} \\ \cline{2-5}
 & $(0.7,0.7,0.1)$ & $(0.72,0.7,0.25)$ & $(1.45,0.7,0.05)$ & $(1,0.953,0.085)$ \\ \hline
0 & \texttimes & \texttimes & \checkmark (for $l \in \llbracket 0, 25 \rrbracket$) & \checkmark (for $l \in \llbracket 0,11 \rrbracket$) \\ \hline
$q_1$ & \checkmark (for $l \in \llbracket0,4\rrbracket$) & \checkmark (for $l\in \llbracket 0,9 \rrbracket$) & \texttimes & \checkmark (for $l \in \llbracket 0, 10 \rrbracket$) \\ \hline
$10^{-1}q_1$ & \texttimes & \checkmark (for $l \in \{6,7\}$) & \texttimes & \checkmark (for $l \in \llbracket 4,7 \rrbracket$) \\ \hline
$10^{-2}q_1$ & \texttimes & \checkmark (for $l=5$) & \checkmark (for $l \in \llbracket 2,18 \rrbracket$) & \checkmark (for $l \in \llbracket 3,11 \rrbracket$) \\ \hline
$10^{-3}q_1$ & \texttimes & \texttimes & \checkmark (for $l \in \llbracket 0,18 \rrbracket$) & \checkmark (for $l \in \llbracket 0,11 \rrbracket$) \\ \hline
$q_2$ & \texttimes & \checkmark (for $l \in \llbracket 0, 9 \rrbracket$) & \checkmark (for $l \in \llbracket11,25 \rrbracket$) & \texttimes \\ \hline
$10^{-1}q_2$ & \texttimes & \checkmark (for $l =6$) & \checkmark (for $l \in \llbracket 4, 25 \rrbracket$) & \texttimes \\ \hline
$10^{-2}q_2$ & \texttimes & \texttimes & \checkmark (for $l \in \llbracket 0, 25 \rrbracket$) & \checkmark (for $l \in \llbracket 5,11 \rrbracket$) \\ \hline
$10^{-3}q_2$ & \texttimes & \texttimes & \checkmark (for $l \in \llbracket 0, 25 \rrbracket$) & \checkmark (for $l \in \llbracket 1,11 \rrbracket$) \\ \hline
\end{tabular}
}
\caption{Validation of the strong cosmic censorship hypothesis for different perturbations of Reissner--Nordström--de Sitter black hole which satisfies the hypothesis.}
\label{tab}
\end{table}

\section{Conclusion}

In this work, we have presented a numerical method to compute resonances for one-dimensional problems, based on a Newton's deflated algorithm.
The main contribution of our approach is that it allows us to determine resonances for compactly supported perturbation of a reference potential $V_0$ as long as one can compute the Jost solutions for the reference potential.
We have shown that using well-chosen reference potentials allows to avoid numerical instability issues observed with a zero reference potential.
Moreover, using this method, we have studied the influence of perturbations on the resonances, and in particular studied their stability in the context of Reissner--Nordström--de Sitter black holes, and provided a numerical study of the strong cosmic censorship hypothesis.
It is natural to wonder whether this stability result can be proven theoretically; this is left for future work.

\section*{Acknowledgements}

The authors would like to thank Nabile Boussaïd and Thierry Daudé for fruitful discussions and comments.
This research was funded in part by the Agence Nationale de la Recherche (ANR), project NUMERIQ (ANR-24-CE46-2255).
The authors acknowledge support from the EIPHI Graduate School (contract ANR-17-EURE-0002).

\printbibliography

\end{document}